\documentclass{amsart}
\usepackage{hyperref}
\usepackage{amsfonts,amscd,amssymb,amsmath,amsthm}
\usepackage{graphicx,caption,subcaption,mathrsfs,appendix}
\usepackage{adjustbox}
\usepackage{nicefrac}
\usepackage{xparse}
\usepackage{bbm}

\usepackage[foot]{amsaddr}

\usepackage{tikz,pgfplots,pgfplotstable}
\usetikzlibrary{arrows,matrix,positioning,fit}
\usetikzlibrary{pgfplots.groupplots}
\usepgfplotslibrary{external}
\begin{document}

\newtheorem{theorem}{Theorem}[section]
\newtheorem{lemma}[theorem]{Lemma}
\newtheorem{corollary}[theorem]{Corollary}
\newtheorem{conjecture}[theorem]{Conjecture}
\newtheorem{cor}[theorem]{Corollary}
\newtheorem{proposition}[theorem]{Proposition}
\theoremstyle{definition}
\newtheorem{definition}[theorem]{Definition}
\newtheorem{assumption}[theorem]{Assumption}
\newtheorem{example}[theorem]{Example}
\newtheorem{claim}[theorem]{Claim}
\newtheorem{remark}[theorem]{Remark}

\newenvironment{pfofthm}[1]
{\par\vskip2\parsep\noindent{\sc Proof of\ #1. }}{{\hfill
$\Box$}
\par\vskip2\parsep}
\newenvironment{pfoflem}[1]
{\par\vskip2\parsep\noindent{\sc Proof of Lemma\ #1. }}{{\hfill
$\Box$}
\par\vskip2\parsep}


\newcommand{\R}{\mathbb{R}}
\newcommand{\T}{\mathcal{T}}
\newcommand{\C}{\mathbb{C}}
\newcommand{\D}{\mathbb{D}}
\newcommand{\G}{\mathcal{G}}
\newcommand{\Z}{\mathbb{Z}}
\newcommand{\Q}{\mathbb{Q}}
\newcommand{\E}{\mathbb E}
\newcommand{\N}{\mathbb N}

\newcommand{\barray}{\begin{eqnarray*}}
\newcommand{\earray}{\end{eqnarray*}}

\newcommand{\Def}{:=}


\DeclareDocumentCommand \Pr { o }
{%
\IfNoValueTF {#1}
{\operatorname{Pr}  }
{\operatorname{Pr}\left[ {#1} \right] }%
}
\newcommand{\Prob}{\Pr}
\newcommand{\Exp}{\mathbb{E}}
\newcommand{\expect}{\mathbb{E}}
\newcommand{\1}{\one}
\newcommand{\Pto}{\overset{\mathbb{P}}{\to} }
\newcommand{\weakto}{\Rightarrow}
\newcommand{\prob}{\Pr}
\newcommand{\pr}{\Pr}
\newcommand{\filt}{\mathscr{F}}
\newcommand{\ohadI}{\mathbbm{1}}
\DeclareDocumentCommand \one { o }
{%
\IfNoValueTF {#1}
{\ohadI }
{\ohadI\left\{ {#1} \right\} }%
}
\newcommand{\sgn}{\operatorname{sgn}}
\newcommand{\Bernoulli}{\operatorname{Bernoulli}}
\newcommand{\Binomial}{\operatorname{Binom}}
\newcommand{\Binom}{\Binomial}
\newcommand{\Poisson}{\operatorname{Poisson}}
\newcommand{\Exponential}{\operatorname{Exp}}

\newcommand{\Var}{\operatorname{Var}}
\newcommand{\Cov}{\operatorname{Cov}}


\newcommand{\Id}{\operatorname{Id}}
\newcommand{\diag}{\operatorname{diag}}
\newcommand{\tr}{\operatorname{tr}}
\newcommand{\proj}{\operatorname{proj}}


\DeclareDocumentCommand \JB { O{n} O{\lambda} } {J_{{#1}}({#2})}

\DeclareDocumentCommand \LP { O{\ESD} } {U_{ {#1} }}
\newcommand{\LPL}{ \LP[{\mu_N}] }

\newcommand{\ESD}{ L_N }

\newcommand{\model}{\mathcal{M}}

\newcommand{\PI}{\Pi}
\DeclareDocumentCommand \PG { O{n} }
{
\mathfrak{S}_{{ #1 }}
}
\newcommand{\RS}{\mathcal{C}}

\newcommand{\TODO}[1]{ {\bf TODO: #1} }
\newcommand{\row}{X}
\newcommand{\col}{Y}
\newcommand{\srow}{x}
\newcommand{\scol}{y}
\newcommand{\csrow}{w}
\newcommand{\cscol}{z}

\newcommand{\COMP}[1]{ \check{#1} }

\newcommand{\AllGOODz}{\mathcal{V}}
\newcommand{\GOODz}{\mathcal{V}_N}
\newcommand{\BADeps}{\epsilon'}

\newcommand{\UBM}{\mathbf{U}}
\newcommand{\UBMd}{ \mathbf{M}}
\newcommand{\UBMn}{ \mathbf{G}}
\newcommand{\Mcol}{ \mathcal{S}}
\newcommand{\UBMcol}{ \mathcal{S}}
\newcommand{\UBTcol}{ \mathcal{T}}
\newcommand{\UBw}{ w_{\mathbf{k}} }
\newcommand{\UBk}{\mathbf{k}}
\newcommand{\half}{\nicefrac{1}{2}}

\newcommand{\UBm}{\mathbf{m}}
\newcommand{\UBconf}{\mathcal{T}}
\DeclareDocumentCommand \NS { O{\UBm} O{\UBk} }{\mathfrak{T}_{{ #1 },{#2}}}

\newcommand{\LBM}{\mathbf{L}}
\newcommand{\LBMd}{ \mathbf{M}}
\newcommand{\LBMn}{ \mathbf{G}}
\newcommand{\LBMcol}{ \mathcal{S}}
\newcommand{\LBTcol}{ \mathcal{T}}
\newcommand{\LBw}{ w_{\mathbf{k}} }
\newcommand{\LBk}{\mathbf{k}}

\title{Regularization of non-normal matrices by Gaussian noise}
\author{Ohad Feldheim}
\address{Department of Mathematics, Weizmann Institute of Science}
\email{ohad\_f@netvision.net.il}
\author{Elliot Paquette}
\address{Department of Mathematics, Weizmann Institute of Science}
\email{paquette@weizmann.ac.il}
\author{Ofer Zeitouni}
\address{Department of Mathematics, Weizmann Institute of Science}
\email{ofer.zeitouni@weizmann.ac.il}
\thanks{Partially supported by an ISF grant and by the Herman P. 
	Taubman chair of Mathematics at the Weizmann Institute.
EP gratefully acknowledges the 
support of NSF Postdoctoral Fellowship DMS-1304057.
}
\date{April 13, 2014}

\begin{abstract}
	We consider the regularization of 
	matrices $M^N$ written in Jordan form
	by additive
	Gaussian noise $N^{-\gamma}G^N$, where $G^N$ is a matrix of i.i.d.
standard Gaussians and $\gamma>\half$ so that the operator norm of the additive
noise tends to $0$ with $N$. Under mild conditions on the structure of $M^N$
we evaluate the limit of the empirical measure of eigenvalues of
$M^N+N^{-\gamma} G^N$ 
and show that it depends on $\gamma$, in contrast with the case of a
single Jordan block.
\end{abstract}

\maketitle

\section{Introduction}
Write $G^N$ for an $N \times N$ random matrix whose entries are i.i.d. Gaussian variables, and
let $\{M^N\}_{N=1}^{\infty}$ be a sequence of deterministic $N \times N$ matrices.
Consider a noisy counterpart 
given by 
\[
  \model^N=M^N+N^{-\gamma}G^N,
\]
where $\gamma\in (0,\infty)$ is fixed.

It is natural to ask how the spectra of $\model^N$ differs from the spectra of $M^N.$
To meaure this, define the \emph{empirical spectral measure} of $\model^N$ as
\[
	L_N^{\model} \Def \frac{1}{N} \sum_{i=1}^N \delta_{\lambda_i}(\model^N)
\]
where $\lambda_i(M^N), i =1,\ldots, N$ are the eigenvalues of $\model^N$, and $\delta_x$ is the Dirac mass at $x$.  

If $\{M^N\}$ are normal, then when $\gamma > \half,$ 
$L_N^{\model} - L_N^{M}$ converges weakly in probability to $0$ as $N \to \infty.$
This follows immediately from Theorem 1.1 of~\cite{Sun96}
and the observation that the operator norm of
$N^{-\gamma}G^N$ goes to $0$ in probability
if and only if $\gamma > \half.$ 
This is one precise sense in which normal matrices are stable under perturbation.  See also~\cite{BhatiaBook} for more background on the stability of the spectra of normal matrices.


The case of non-normal matrices is more complicated.
To illustrate a particular well-known example,
consider the $N \times N$ nilpotent matrix
\[
	T^N =
	\begin{bmatrix}
		0  & 1 &   &   &    \\
		   & 0 & 1 &   &    \\
		   &   &\ddots & \ddots  &   \\
		   &   &   & 0 & 1  \\
		   &   &   &   & 0  \\
	\end{bmatrix}.
\]
This is a matrix whose eigenvalues are highly sensitive to perturbation.  Indeed, adding $\epsilon$ to the lower left entry of the matrix makes the eigenvalues of the perturbed matrix distribute as the $N$-th roots of $(-1)^{N+1}\epsilon.$  Thus, for $N$ large, any polynomially small (in $N$) perturbation $\epsilon$ will cause all the eigenvalues to move nearly unit distance.  For a general discussion of spectral instability of non-normal matrices and links to the notion of pseudospectra, see the comprehensive treatise~\cite{Trefethen}.

When Gaussian noise is added to $T^N,$ it is a consequence of \cite{Guionnet14} that $T^N + N^{-\gamma} G^N$ for $\gamma > \half$ has empirical spectral measure converging weakly to the uniform distribution on the unit circle.  
One way of explaining why this limiting distribution appears is through the notion of $\ast$-moment convergence.

Recall that a sequence of matrices $M^N$ converges in $\ast$-moments to an element $a$ in a $W^*$ probability space $(\mathcal{A},\|\cdot\|,*,\phi)$ if for any non-commutative polynomial $P$
\[
  \frac{1}{N} \tr P(M^N, \operatorname{Adj}({M^N})) \to \phi( P(a,a^*))
\]
as $N \to \infty$ (see~\cite[Chapter 5]{AnGuZe} for the necessary background on $W^*$ probability spaces).
In the case of the nilpotent matrices $T^N,$ these converge in $\ast$-moments to the Haar unitary element of $\mathcal{A},$ i.e. they have the same $\ast$-moment limit as a sequence of $N \times N$ Haar distributed unitary matrices.   The spectral measure of this limiting operator is the uniform measure on the unit circle.  Note that this example shows that convergence of $\ast$-moments 
does not imply the convergence of 
the corresponding 
empirical eigenvalue measure, which in this case is just $L_N^T = \delta_0.$

The results of \cite{Guionnet14},
which quantify some of the statements in \cite{Sniady02},
show that under appropriate assumptions on $a$ and $M^N$,
the 
empirical measure
$L_N^\model$ does converge to the spectral measure of the limiting
operator $a.$ (See~\cite{Guionnet14} for precise statements.)  Thus, in a
sense, the spectra of the limiting operator $a$ accurately represents
the spectra of its finite dimensional counterparts.  Note that in the setup
considered in \cite{Guionnet14}, the limit is independent of $\gamma,$ provided that $\gamma > \half.$

Here, we are concerned with the situation in which the limiting picture fails to accurately represent the spectra of the finite-dimensional random matrices.  For example, consider the $N \times N$ matrix
\begin{equation}\label{eq:tbn}
	A_{b}^N = \begin{bmatrix}
		T^b &   &    &   &  \\
		    &T^b&    &   &  \\
		    &   &\ddots& &  \\
		    &   &    &T^b&  \\
		    &   &    &   & T^c\\
	\end{bmatrix},
\end{equation}
where the $c \leq b.$
The matrix $A_{\log N}^N$ still converges in $\ast$-moments to the Haar unitary element $a,$ but it is shown in Proposition 7 of~\cite{Guionnet14} that
the limsup of the spectral radius $A_{\log N}^N+N^{-\gamma} G^N$ is strictly
smaller than $1$, if $\gamma>\gamma_0$ for some fixed $\gamma_0$.

In this paper, we show that a natural class of matrices 
generalizing
$A_{\log N}^N$, when perturbed by Gaussian noise $N^{-\gamma}G^N$, have
empirical measures of eigenvalues 
converging to $\gamma$-dependent limits.

\subsection*{Definitions and main results}

For each $N$, let $\{B^i(N)\}_{i=1}^{\ell(N)}$ be a family of Jordan blocks, 
with $B^i=B^i(N)$ having dimension $a_i\log N$ where $a_i=a_i(N)$ and 
$\sum_{i=1}^{\ell(N)} a_i\log N=N$.
We denote by $c_i=c_i(N)$ the eigenvalue of $B^i$.

Introduce the matrix
\[
M^N = \begin{bmatrix}
B^1 & & & \\
& B^2 & & \\
& & \ddots & \\
& & & B^{\ell(N)} \\
\end{bmatrix}.
\]
Fix $\gamma > \half$ and consider the matrix $\model = M^N + N^{-\gamma}
G^N,$ where $G^N$ is a matrix of i.i.d. standard normal variables.  Our
main result gives the convergence of the empirical distribution of
eigenvalues of $\model.$

To describe the limit, let 
$r_i = r_i(N)= \exp(\nicefrac{(-\gamma+\half)}{a_i}).$ Let
$m_{c,r}$ be the uniform probability measure on the circle centered at $c$ with radius
$r.$  Let $\ESD$ denote the empirical spectral measure of $M^N.$ Set $\mu_N$
to be the measure on $\C$ given by
\[
	\mu_N \Def \sum_{i=1}^{\ell(N)} \frac{a_i\log N}{N}m_{c_i,r_i}.
\]

If $\gamma > 1,$  we show that $\mu_N$ converges to $\mu$ provided $\ell = o(N).$

\begin{theorem}
\label{thm:main0} Suppose that $\gamma > 1$ and $\ell(N)=o(N).$
Suppose further that there is a compact $K \subseteq \C$ so that all
$\mu_N$ are supported on $K,$ and suppose that there is a probability measure
$\mu$ so that $\mu_N \weakto \mu.$
Then, the empirical measure $\ESD$ of $M^N + N^{-\gamma} G^N$
converges to $\mu$ weakly in probability.
\end{theorem}
\begin{remark}
	\label{remark-ofer1}
	If the sequence $\{\mu_N\}$ is  tight but not necessarily
	convergent, one could
	rephrase Theorem \ref{thm:main0} as the statement that
	$d(L_N,\mu_N)\to_{N\to\infty} 0$ in probability,
	where $d$ is any metric compatible
	with weak convergence.
\end{remark}

Theorem~\ref{thm:main0} is an immediate consequence of
our main result, Theorem~\ref{thm:main1} below, which
handles also the case $\gamma\in (1/2,1]$ at the cost of imposing
an extra condition, 
essentially
that circles
arising from polynomially large blocks 
cover only a small portion of the plane;
we now make this extra condition precise.

Fix $\BADeps > 0$ satisfying $\BADeps < 2\gamma-1.$ Set 
\[
	\GOODz = \GOODz(\BADeps)=\left\{ z \in \C: \forall i \in [\ell(N)],
	\min(a_i \log(N),|1-|z-c_i|^2|^{-1}) < N^{2\gamma - 1 - \BADeps}
	\right\}
\]
and
  \[
	  \AllGOODz=\AllGOODz(\BADeps)
	  \Def\bigcup_{N=k}^\infty \bigcap_{k=1}^\infty  \GOODz.
  \]

\begin{assumption}
  \label{ass:small_area}
  There exists $\BADeps\in (0,2\gamma-1)$ so that
   $\C\setminus \AllGOODz(\epsilon')$ has Lebesgue measure $0.$
\end{assumption}
Note that Assumption \ref{ass:small_area} trivially holds when $\gamma>1$ by
choosing $\BADeps=\gamma-1$.
Our main result holds in the regime $\gamma>1/2$ under
Assumption 
\ref{ass:small_area}.

\begin{theorem}
\label{thm:main1}
Suppose that $\gamma > \half$, $\ell(N)=o(N)$  and
Assumption~\ref{ass:small_area} holds. Suppose further that there is a compact $K
\subseteq \C$ so that all $\mu_N$ are supported on $K,$ and suppose that
there is a probability measure $\mu$ so that $\mu_N \weakto \mu.$ Then the
empirical measure $\ESD$ of $M^N + N^{-\gamma} G^N$ converges to
$\mu$ weakly in probability.
\end{theorem}
As mentioned before, Assumption \ref{ass:small_area} holds automatically when 
$\gamma>1$, and therefore 
Theorem~\ref{thm:main0} follow immediately from
Theorem~\ref{thm:main1}.
An illustration of the $\gamma$-dependency 
in Theorem \ref{thm:main1} is given in Figure~\ref{fig:1}.

Another case in which $\C\setminus\AllGOODz$ has Lebesgue measure 
$0$ is when all the
eigenvalues of $M^N$ are the same.

\begin{corollary}
\label{cor:main0.5} Suppose that $\gamma > \half$, that
$\ell(N)=o(N)$ and that $c_i=c_1$, $i=2,\dots,\ell$.
Suppose further that there is a compact $K \subseteq \C$ so that
all $\mu_N$ are supported on $K,$ and that 
$\mu_N \weakto \mu$
for some
probability
measure $\mu$.
Then, the empirical measure $\ESD^\model$ of $M_N + N^{-\gamma} G_{N}$ 
converges to
$\mu$ weakly in probability.
\end{corollary}
In this case, one may check that $\C\setminus \AllGOODz$ is in fact contained in the circle of radius one centered at $c_1.$ See Figure~\ref{fig:2} for an illustration. 

\begin{figure}[t]
\begin{center}
\begin{tikzpicture}
\begin{axis}[
xlabel = {$\Re z$},
ymin = -2, ymax = 1,
xmin = -2, xmax = 2,
width = 0.55\textwidth,
scale mode=scale uniformly,
legend entries={$\gamma = 1.0$}
]
\addplot+[opacity=0, fill opacity=1.0, only marks,mark size=0.4pt] table {./RingsReal1_5000.dat};
\end{axis}
\end{tikzpicture}
\begin{tikzpicture}
\begin{axis}[
xlabel = {$\Re z$},
ymin = -2, ymax = 1,
xmin = -2, xmax = 2,
yticklabel=\ ,
width = 0.55\textwidth,
scale mode=scale uniformly,
legend entries={$\gamma = 0.8$}
]
\addplot+[opacity=0, fill opacity=1.0, only marks,mark size=0.4pt] table {./RingsReal08_5000.dat};
\end{axis}
\end{tikzpicture}
\caption{
 The eigenvalues of a deterministic $5000\times 5000$ matrix perturbed by two different magnitudes of noise.
 The matrix consists of $5$ types of blocks centered at $-1,0,1,-0.5-0.8i,\text{and } 0.5-0.8i.$ The sum of the dimensions of the blocks for each type is roughly the same.  Each individual block is of dimension $\lceil\log 5000\rceil = 9.$
Note the finite $N$ effects.}
\label{fig:1}
\end{center}
\end{figure}
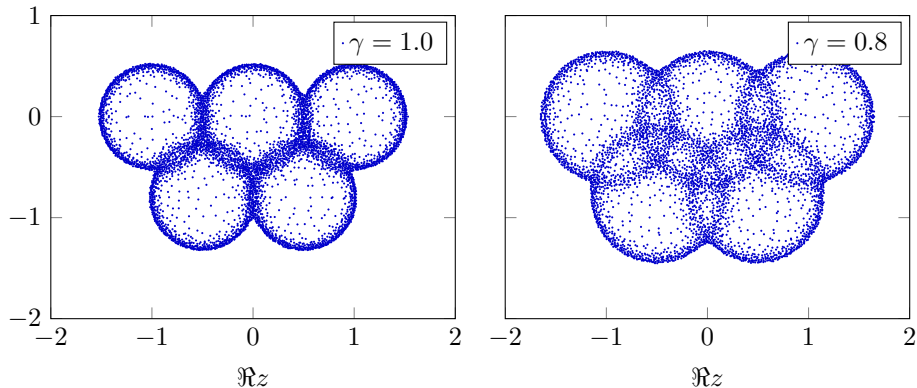

\subsection*{Discussion}

Theorem \ref{thm:main1} shows that $A_{\log N}^N + N^{-\gamma}G^N$ has
empirical eigenvalue distribution converging to a measure which has mass $1$
uniformly distributed on the circle of radius $e^{-\gamma+\half}.$  In particular,
this does not agree with what would be seen if the blocks were perturbed
separately.  If each $T^b$ in \eqref{eq:tbn} were replaced by $T^b +
N^{-\gamma} G^b,$ then with $b = \log N(1+o(1)),$ the resulting matrix
would have eigenvalue distribution converging to a circle of radius
$e^{-\gamma}.$ 
Thus,
applying noise only to the diagonal blocks of
$A^N_{\log N}$ does not
make the matrix insensitive to further perturbation of the off-diagonal entries.

\pgfplotsset{
    every non boxed x axis/.style={} 
}
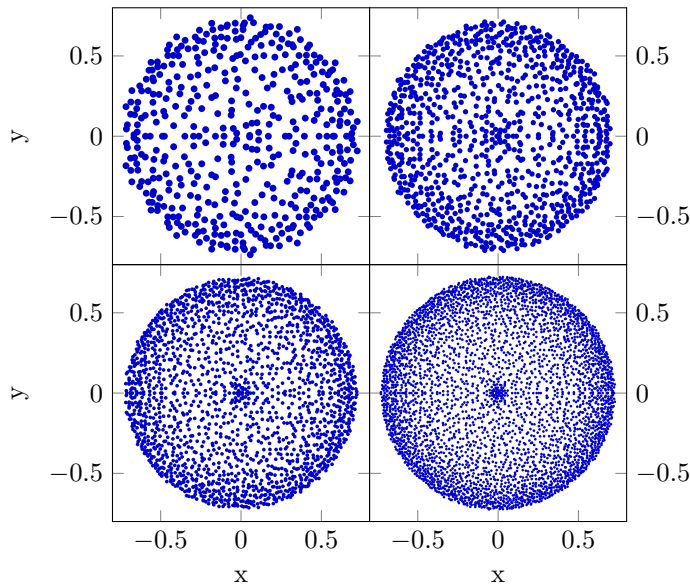
\begin{figure}[t]
\begin{center}
\begin{tikzpicture}
\begin{groupplot}[
group style={
group size=2 by 2,
xticklabels at=edge bottom,
horizontal sep=0pt,
vertical sep=0pt,
xlabels at=edge bottom,
ylabels at=edge left
},
width=6cm,
xmin = 0, xmax = 1,
xlabel=x,
ylabel=y
]

\nextgroupplot[
  	ymin=-0.8,ymax=0.8,
  	xmin=-0.8,xmax=0.8,
	scale mode=scale uniformly,
	height=5.0cm]
\addplot+[blue, opacity=0, fill opacity=1.0, only marks, mark size=1.3pt, samples=50] table {./DensityReal075_500.dat};

\nextgroupplot[
  	ymin=-0.8,ymax=0.8,
  	xmin=-0.8,xmax=0.8,
	scale mode=scale uniformly,
	yticklabel pos=right,
	height=5.0cm]
\addplot+[blue, opacity=0, fill opacity=1.0, only marks, mark size=1.0pt, samples=50] table {./DensityReal075_1000.dat};

\nextgroupplot[
  	ymin=-0.8,ymax=0.8,
  	xmin=-0.8,xmax=0.8,
	scale mode=scale uniformly,
	height=5.0cm]
\addplot+[black, opacity=0, fill opacity=1.0, only marks, mark size=0.700pt, samples=50] table {./DensityReal075_2000.dat};

\nextgroupplot[
  	ymin=-0.8,ymax=0.8,
  	xmin=-0.8,xmax=0.8,
	scale mode=scale uniformly,
	yticklabel pos=right,
	height=5.0cm]]
\addplot+[black, opacity=0, fill opacity=1.0, only marks, mark size=0.5pt, samples=50] table {./DensityReal075_4000.dat};

\end{groupplot}

\end{tikzpicture}
\caption{
For various values of $N,$ set $M^N$ to be the matrix with all eigenvalues equal to $0$ and approximately equal number of blocks of size $0,1,2,\ldots, \lceil \log N \rceil.$  
From left-to-right and top-to-bottom, the eigenvalues of $M^N + N^{-\gamma}G^N$ are given
for $\gamma=\nicefrac{3}{4}$ and $N=500,1000,2000,\text{ and }4000.$  The limiting density is given by $\nicefrac{-C}{r \log r}\one[r \leq e^{-\nicefrac{1}{4}}]$ for normalizing constant $C.$ 
}
\label{fig:2}
\end{center}

\end{figure}

While Theorem \ref{thm:main1} covers many choices of Jordan blocks, it does
put technical limitations on the types of
matrices and noise that can be used. In particular it would be of interest
to remove the restriction on $\ell(N).$
As in \cite{Wood14},
the Gaussian assumption on the noise probably could also be weakened,
though this requires a better understanding of how small $\det(G + C)$ can be,
where $G$ is a matrix of i.i.d. elements and $C$ is some arbitrary matrix
(in particular, without a priori estimates on the norm of $C$); 
such control is
not currently available 
for small singular
values of $G+C$ without putting some a-priori conditions on $C$,
see e.g. 
\cite{Tao10} for the case of minimal singular value.

Far beyond these possible extensions, it would be interesting to generalize
Theorem~\ref{thm:main1} to cover matrices that are not in Jordan form,
i.e. proving a theorem about the noise perturbation of
$S^N M^N (S^N)^{-1}$ for $M^N$ in Jordan
form.
This however seems to require some constraints
on the sequence $S^N$ so that they do not become
progressively ill-conditioned going down the sequence.


\subsection*{Proof approach}

The approach is based on the in-probability convergence of the logarithmic potential of $\ESD$ to the logarithmic potential of the corresponding measure,
which appears frequently in the study of non-normal random matrices (see~\cite{Chafai12}).  For a compactly supported probability measure $\rho $ on $\C,$ define
\(
\LP[\rho](z) = \int_{\C} \log | z- x|\, d\rho(x).
\)
Note the logarithmic potential of $\ESD$ can also be expressed as
\(
\LP(z) = \frac{1}{N} \log | \det( \model - zI ) |,
\)
where $I$ is the identity matrix.

To show the desired convergence of $\ESD$ to $\mu,$ it suffices to show that:
\begin{itemize}
\item[a)] 
	There is a compact $K \subseteq \C$ so that for all $\epsilon >0,$
		\(
		\Pr[ \ESD( K^c) > \epsilon] \to 0.
		\)
	\item[b)] For almost every $z \in \C,$ $\LP(z) \Pto \LP[\mu](z).$
\end{itemize}
For a proof, see Theorem 2.8.3 of~\cite{TaoBook}. The tightness
condition a) 
is standard, and quickly follows from the assumed compact support of the collection $\{\mu_N\}.$ Thus, one needs to checks the convergence in b). Toward
this end, we first discuss the convergence of $\LP[\mu_N]$
to $\LP[\mu]$.

Since
$\mu$ is a probability measure, 
$\LP[\mu](z)\in
L^p_{\operatorname{loc}}$ for each 
$1 \leq p < \infty $. In particular,
$\LP[\mu](\cdot)$ 
is finite almost everywhere.
Together with the existence of the compact $K$ that contains the support of
the different $\mu_N$, we also have the uniform integrability of
$\LP[\mu_N]$ on compact subsets of $\C$. Together with the 
weak convergence $\mu_N\weakto \mu$, this implies 
that $\LP[\mu_N]\to \LP[\mu]$ in $L^p_{\operatorname{loc}}$.
Passing to subsequences if necessary, we deduce the convergence 
of $\LP[\mu_N](z)$ to $\LP[\mu](z)$ for Lebesgue almost every $z$.
%
Thus, the proof of b), and therefore of
Theorem~\ref{thm:main1}, is reduced to showing
\begin{align}
\label{eq:LPL}
\forall z \in \AllGOODz\ : |\LP(z) - \LPL(z)| \Pto 0.
\end{align}

We obtain the convergence in~\eqref{eq:LPL} by showing upper and lower bounds on
$\LP(z)$.  The upper bound is obtained through a careful expansion of the
determinant of $\model - zI = M^N + N^{-\gamma}G^N - zI$ as a linear
combination of the minors of $N^{-\gamma}G^N.$  The minors of $G^N$
are then bounded by a relatively crude union bound (see Lemma~\ref{lem:ub_noise}),
and the sum is estimated by a leading order term analysis.

The upper bound argument works with significantly weaker assumptions than Theorem~\ref{thm:main1}.
In particular Assumption~\ref{ass:small_area} is not used at all. Furthermore,
it should be straightforward to weaken the assumptions on the noise to include entries whose distributions are either non-Gaussian i.i.d., or Gaussian with non-trivial covariance
matrix.

The lower bound, on the other hand, is more delicate.  Here we first apply a sequence of row and column permutations to the matrix to put it in the form
\[
	\model - zI =
	\begin{bmatrix}
		A +G_1 & * \\
		* & G_2 \\
	\end{bmatrix},
\]
(see~\eqref{eq:LBblock}), where $G_1$ and $G_2$ are pure noise matrices and $A$ is stable with respect to Gaussian perturbation.  This representation allows us to compute the determinant
by the Schur complement formula, whose general form 
\[
	\det(\model - zI) = \det(A+G_1)\det(G_2 - C),
\]
where $C$ is some matrix.

As $A$ is stable with respect to Gaussian perturbation, we use a second moment computation
to show that the determinant of $A+G_1$ is a 
good approximation of the determinant of $A$.
By showing that $|\det (G_2 - C)|$ stochastically dominates $|\det G_2|,$
we are able to obviate understanding $C.$  In this step we crucially
use the Gaussian assumption on the matrix, and we believe this portion
of the argument is the largest obstruction to proving the theorem for more
general noise models.

Assumption~\ref{ass:small_area} is necessary for the second moment
estimate.  As can be seen from calculating the variance of $\det(I + zT^N +
N^{-\gamma}G^{N})$ for various $z$ with $|z| < 1$ and $\half <
\gamma \le 1,$ if $z$ is very close $1$ (going to $1$ at some polynomial rate),
the variance can be made to grow to infinity, while the expectation is $1$.
This phenomenon disappears when $\gamma > 1,$ for which reason
Assumption~\ref{ass:small_area} is vacuous for these $\gamma.$
Thus for the purpose of showing $\det(A+G_1) \approx
\det(A),$
the second moment method is an insufficient tool when
$\half < \gamma \le 1.$ It is unclear whether
Assumption~\ref{ass:small_area}
could be weakened or completely omitted by
applying other methods of proof.

\subsection*{Organization}
This paper is organized into 5 sections.
In Section \ref{sec:prelim}, we establish notation that we use throughout the paper as well as many relevant calculations and lemmata that we need for the upper and lower bounds.
In Section \ref{sec:ub}, we show the upper bound on the log potential, and in Section \ref{sec:lb} we show the lower bound.
In Section \ref{sec:final}, we give the proof of Theorem~\ref{thm:main1}.

\section{Preliminaries}
\label{sec:prelim}

In this section we present notation and auxiliary lemmata that are used to simplify the rest of the paper. These are divided according to their general topic. Throughout the paper, whenever we state that a property holds with \emph{high probability} this is meant to say that the probability tends to $1$ as $N$ tends to infinity.

\subsection*{Log potential}
For a natural number $k,$ we let $[k] \Def \{1,2,3,\ldots,k\}.$
To simplify our calculations and definitions, we set
$$\nu\Def\gamma - \half.$$
Also, we often omit the dependence of parameters on $N.$

For each $i \in [\ell],$ define
\[
g_i \Def (-\nu) - a_i \log|z-c_i|.
\]
This allows the log potential $\LPL(z)$ to be given by
\begin{equation}
  \label{eq:target}
  \frac{N}{\log N}\LPL(z) = \sum_{i : g_i \leq 0} a_i \log |z - c_i|-\sum_{i : g_i > 0} \nu.
\end{equation}
Note that the expression is continuous in $z.$
\subsection*{Matrix decomposition}
For an $N\times N$ matrix $A$, and $\row,\col\subseteq [N]$ we write $A[X,Y]$ for the submatrix of $A$ which consists of the rows in $\row$ and the columns in $\col$.

Our goal is to provide upper and lower bounds on $\frac{1}{N} \log |\det(\model -zI)|$ which are arbitrarily close to $\LPL(z)$ for large $N$.
In the process 
of obtaining both bounds, $\model - zI$ is decomposed into a sum of matrices $A$ and $B.$
In general, we may expand the determinant of the sum of two $N\times N$ matrices $A$ and $B$ as
\begin{equation}\label{eq:det_decomposition}
	\det(A+B) = \sum_{\substack{\row,\col \subset [N] \\ |\row|=|\col|}} (-1)^{\sgn(\sigma_\row)\sgn(\sigma_\col)} \det(A[\COMP{\row};\COMP{\col}])\det(B[\row;\col]),
\end{equation}
where $\COMP{\row}:=[N]\setminus \row$, $\COMP{\col}:=[N] \setminus \col$ and $\sigma_Z$ for $Z\in\{\row,\col\}$ is the permutation in $S_N$ which places all the elements of $Z$ before all the elements of $\COMP{Z}$, but preserves the order of elements within the two sets. In particular, observe that the notation $A[\COMP{\row} ; \COMP{\col}]$ denotes the submatrix of $A$ given by deleting the rows in $\row$ and the columns in $\col$.

In our application $A$ will be an \emph{upper bi-diagonal block matrix} with $\ell$ blocks, that is,
a block matrix whose non-zero entries lie on the main diagonal and the first superdiagonal. For such a block matrix,
$A_i$ is used to denote the $i$-th block.  Likewise, $\row_i,\col_i$ are used to denote the rows and columns of
$A_i$ that are contained in $\row$ and $\col$ respectively.
%
This notation allows the decomposition of the determinant of a submatrix of a block matrix as a product of determinants of small matrices.
Here and in the rest of our formulae we always assume the determinant of the matrix of size $0$ to be one.

\begin{lemma}
  \label{lem:unbalanced_minors}
Let $A$ be an $N\times N$ block matrix with $\ell$ blocks, and let $\row,\col\subset [N]$ be such that $|\row|=|\col|$. We have
  \[
  \det(A[\row;\col]) = \begin{cases} \prod_{i=1}^\ell \det( A_i[\row_i; \col_i]) & \forall i\in[\ell]\ :\ |\row_i|=|\col_i|\\
                                     0 & \text{otherwise}\end{cases}.
  \]
\end{lemma}
\begin{proof}
We let $k=|\row|=|\col|$. Suppose the second case holds.
By permuting the order of blocks, we may assume without loss of generality $|\row_1|>|\col_1|$. Expanding the determinant of $U[\row;\col]$ by the Leibniz formula to get
\[\det( A[\row; \col])=\sum_{\sigma\in S_{k}}\sgn(\sigma)\prod_{i=1}^{k}A[\row; \col]_{i,\sigma(i)},\]
we observe that by the pigeonhole principle every $\sigma\in S_{k}$ must satisfy $\sigma(i)>|\col_1|$ for some $i\le|\row_1|$. For this $i$ we have $A[\row; \col]_{i,\sigma(i)}=0,$
and therefore the entire product nullifies. In the first case where $|X_i|=|Y_i|$ for all $i$, it is straightforward to check that the determinant is that of a block matrix,
and thus it is the product of the block determinants.
\end{proof}

Additionally, when $A$ is an upper bi-diagonal block matrix, we can further simplify Lemma~\ref{lem:unbalanced_minors}.
\begin{lemma}
  \label{lem:bidiag_det}
Let $A$ be an $N\times N$ block matrix with $\ell$ blocks, and let $\row,\col\subset [N]$ be such that $|\row|=|\col|$. If in addition $A$ is upper bi-diagonal, then
$\det(A[\row;\col])$ is given by the product of its diagonal entries.
\end{lemma}
\begin{proof}

   We let $k=|\row|=|\col|$ and write
   \begin{align*}
     \row &= \left\{ \srow_1<\srow_2<\ldots<\srow_k \right\} \text{and} \\
     \col & =\left\{ \scol_1<\scol_2<\ldots< \scol_k \right\}.
   \end{align*}
  Expand the determinant of $A[\row;\col]$ by the Leibniz formula to get
 \[\det( A[\row; \col])=\sum_{\sigma\in S_{k}}\sgn(\sigma)\prod_{i=1}^{k}A[\row; \col]_{i,\sigma(i)},\]

The claim is equivalent to showing that for all $\sigma$ not equal to the identity,
  \(
  \prod_{i=1}^k A_{\srow_i, \scol_{\sigma(i)}} = 0.
  \)
  From bi-diagonality, if any $\scol_{\sigma(i)} \notin \{\srow_i, \srow_i+1\},$ then
  $A_{\srow_i, \scol_{\sigma(i)}}$ is $0,$ and the claim is complete.  Thus, we may restrict ourselves to $\scol_{\sigma(i)} \in \{\srow_i, \srow_i+1\}$ for all $i\in[k]$.
  Since $\srow_i$ and $\scol_i$  are strictly increasing, we deduce that
 $\sigma(1) \leq \sigma(2) \leq \cdots \leq \sigma(k).$  As $\sigma$ is a permutation, this forces all these inequalities to be strict and hence $\sigma$ is the identity.  Thus only the identity permutation can possibly have $\prod_{i=1}^k A_{\srow_i, \sigma(\scol_i)} \neq 0.$
\end{proof}

We conclude this part by presenting yet another simplification of the formula for $\det(U[\COMP{\row};\COMP{\col}])$
when $U$ is a single block of the form $I + zT^N.$

\begin{lemma}
\label{lem:block_calc}
Let $U = I + zT^N.$
For $\row,\col \subseteq [N]$ with $|\row|=|\col|=k,$ write
$\row = \left\{ \srow_1<\srow_2<\ldots<\srow_k \right\}$ and $\col=\left\{ \scol_1<\scol_2<\ldots< \scol_k \right\}$.  Then
\[
\det(U[\COMP{\row};\COMP{\col}]) =
\prod_{i=1}^k z^{\srow_i - \scol_i}
\one[ \scol_i \le \srow_i < \scol_{i+1}, \forall i, 1 \leq i \leq k],
\]
where we take $\scol_{k+1} = \infty.$
\end{lemma}
\begin{proof}
Write $\COMP{\row} = \left\{ \csrow_1<\csrow_2<\ldots<\csrow_{N-k} \right\}$ and $\COMP{\col} = \left\{ \cscol_1<\cscol_2<\ldots< \cscol_{N-k} \right\}.$
Using Lemma~\ref{lem:bidiag_det} we observe that
\begin{equation}
\label{eq:detprod}
  \det(U[\COMP{\row};\COMP{\col}]) = \prod_{i=1}^{N-k} U_{\csrow_i, \cscol_i}.
\end{equation}

Observe that by the bi-diagonal structure of $U$, this product
nullifies unless \begin{equation}
  \label{eq:rowcol_condition}
  \cscol_r \in \{\csrow_r, \csrow_{r}+1\}\cap \COMP{\row}.
\end{equation}
For $i\in[k+1]$ write $U_i=\{\srow_{i-1}+1,\srow_{i-1}+2,\dots,\srow_{i}\}$ setting $\srow_{0}=0$ and $\srow_{k+1}=N$.
Since $\srow_i\in\row$ we have by \eqref{eq:rowcol_condition} that unless
$\cscol_r\in U_i$ for all $r$ satisfying $\csrow_r\in U_i$, the product \eqref{eq:detprod} nullifies.
Since for all $i\in[k]$ we have $|\{r\ :\ \csrow_r\in U_i\}|=|U_i|-1$ and $|\{r\ :\ \csrow_r\in U_{k+1}\}|=|U_{k+1}|$, we conclude that each $U_i$ with $1 \le i\le k$ contains exactly one element of $\col$, and the $k+1$
block contains none of them.
This implies that unless
\[
  \scol_1 \leq \srow_1 \lneq \scol_2 \leq \srow_2 \lneq \cdots \lneq \scol_k \leq \srow_k,
\]
the product \eqref{eq:detprod} nullifies.
The stated formula for the determinant now follows by noting that for a given block $U_j$
\[
  \prod_{i : \srow_i \in U_j}  U_{\srow_i, \scol_i} = z^{\srow_j - \scol_j}.
\]
\end{proof}

\subsection*{Gaussian estimates}
\label{sec:Gdet}
In this section, we present several lemmata involving
estimates on the determinant of a Gaussian matrix. All the bounds in this
section are based on the following identity in law for an $N \times N$ matrix
$E$ of independent standard Gaussians, see \cite{Goodman},
\[
  |\det E|^2 \overset{\mathcal{L}}{=} \prod_{r=1}^N \chi_r^2,
\]
where $\chi_r$ are independent and have the distribution of the length of an $r$-dimensional standard Gaussian vector.
For $t > -\tfrac r 2,$ we have the following moment formula for the $\chi_r$ variable:
\begin{equation}
  \label{eq:chi_moment}
  \Exp \chi_r^{2t} = 2^{t} \frac{\Gamma( \tfrac{r}{2} + t)}{\Gamma( \tfrac{r}{2})}.
\end{equation}

\begin{lemma}
  \label{lem:Gdet_ub}
  Let $E$ be an $k \times k$ matrix of independent standard Gaussians.  For any $\delta >0,$ the following 
  holds,
  \[
    \Pr[
      \det|E| \geq (k!)^{\half + \delta}
    ] \leq \exp\left(-\lfloor k\delta\rfloor^2(\log(\nicefrac{k}{2e^2})\right)
  \]
\end{lemma}
\begin{proof}
  For natural numbers $r$ and $t,$ the $\chi_r$ moment formula~\eqref{eq:chi_moment} simplifies to
  \[
    \Exp \chi_r^{2t} = r(r+2)(r+4)\cdots (r+2t-2).
  \]
  Thus for some natural $t,$ the moment of the determinant can be given by
  \[
    \Exp |\det E|^{2t} = \prod_{r=1}^k \prod_{i=0}^{t-1} (r+2i)
    \leq k! (k+2)! (k+4)! \cdots (k + 2t - 2)!.
  \]
  For $t \leq N,$ this can be bounded by
  \[
    \Exp |\det E|^{2t} \leq (k!)^t (2k)^{t^2}.
  \]
  By Markov's inequality, we therefore have that
  \[
    \Pr[  \det|E| \geq (k!)^{\half + \delta} ]
    \leq \frac{(k!)^{t} (2k)^{t^2}}{(k!)^{t + 2\delta}}
  \]
  for any integer $t \leq k.$  Using that $k! \geq k^ke^{-k},$ we get that
\begin{align*}
    \Pr[  \det|E| \geq (k!)^{\half + \delta} ]
    &\leq
    \exp\left( t^2 \log (2k) - 2t\delta k \log (\nicefrac{k}{e})\right)\\
    &= \exp\left( t^2 \log (2k) - t\delta k \log (\nicefrac{k^2}{e^2})\right)
\end{align*}
  for any integer $t \leq k.$  Taking $t = \lfloor k\delta \rfloor,$ we get
\begin{align*}
\Pr[\det|E| \geq (k!)^{\half + \delta}] &\leq
    \exp\left( \lfloor k\delta \rfloor^2 \log (2k) - \lfloor k\delta \rfloor\delta k \log (\nicefrac{k^2}{e^2})\right) \\
    &\le \exp\left( \lfloor k\delta \rfloor^2 \log (\nicefrac{2e^2}{k})\right).
\end{align*}

\end{proof}

\begin{lemma}
  \label{lem:Gdet_lb}
  Let $E$ be an $N\times N$ matrix of independent standard Gaussians.  Then there are constants $c_1 >0$ and $c_2>0$ so that
  \[
    \Pr[
      \det|E| \leq \sqrt{N!} e^{-c_1 N}
    ] \leq \frac{1}{c_2}e^{-c_2N}.
  \]
\end{lemma}
\begin{proof}
  We will use negative moments and Markov's inequality to get the desired bound.  Fix an integer $K>0.$  Note that
  \[
    F_K \Def \prod_{r=1}^k \chi_r
  \]
  is absolutely continuous and has a bounded density.  Thus, the probability that $F_K < 2^{-N}$ is exponentially small in $N.$

  To prove the statement of the lemma, it therefore suffices consider
  \(
  L_{K} \Def \prod_{r=K+1}^N \chi_r.
  \)
  For this variable, we need to show that there are constants $c_i>0$ so that for $N > K,$
  \[
    \Pr[
      L_K \leq \sqrt{N!} e^{-c_1 N}
    ] \leq \frac{1}{c_2}e^{-c_2N}.
  \]

  Now by~\eqref{eq:chi_moment}
  \[
    \Exp \chi_r^{-2} = \frac{1}{(r-2)}.
  \]
  Thus for some $K$ sufficiently large, we get that for all $r > K,$
  \[
    \Exp \chi_r^{-2} \leq \frac{2}{r}.
  \]
  Hence for this $K,$
  \[
    \Exp L_K^{-2} \leq K! \frac{2^{N}}{N!}.
  \]
  Applying Markov's inequality
  we get
  \[
    \Pr[
      L_K \leq \sqrt{N!} 2^{-N}
    ]
    =
    \Pr[
      L_K^{-2} \geq \frac{4^{N}}{N!}
    ]
    \leq K!{2^{-N}},
  \]
  completing the proof.
\end{proof}

\section{Upper bound}
\label{sec:ub}
This section is dedicated to the proof of the following proposition.
\begin{proposition}
  \label{prop:ub}
  For $\ell=o(N)$ and any $z \in \C,$ we have that for all $\delta > 0,$
  \[
    \Pr[
      \LP(z) \leq \LPL(z) - \delta
    ] \to 0.
  \]
\end{proposition}

Let $\Mcol$ denote the collection of blocks so that $|z-c_i| \geq 1.$ Define
$\UBM$ to be a modification of the matrix $\model-zI$ in which each column
intersecting a block from $\Mcol$ is scaled by $|z-c_i|^{-1}.$ This implies
the following relationship:
\begin{equation}
	|\det(\model - zI)| = |\det \UBM| \prod_{i \in \Mcol} |z-c_i|^{a_i \log N}
	\label{eq:ub0}
\end{equation}
Decompose $\UBM$ as a sum of $\UBMd + \UBMn,$ where $\UBMd = \Exp \UBM.$ This
gives $\UBMd$ the same block structure as $\model.$ 

Write $\UBk = (k_i)_{i \in [\ell]}$ for an element of the hypercube
$\{0,1\}^\ell.$ We define $\PI_{\UBk}$ as the subset of $\{(\row,\col),
|\row|=|\col|\}$ which satisfies:
\begin{enumerate}
  \item For each $i \in [\ell],$ $|\row_i| = |\col_i|.$
  \item For each $i \in [\ell],$ $|\row_i| > 0$ if and only if $k_i = 1.$
\end{enumerate}
Combining this together with Lemma~\ref{lem:unbalanced_minors} and
\eqref{eq:det_decomposition} we have
\begin{equation*}
  \det(\UBM) = \sum_{\UBk \in \{0,1\}^\ell} \sum_{(\row,\col) \in \PI_{\UBk}} (-1)^{s(\row,\col)} \det(\UBMd[\COMP{\row};\COMP{\col}])\det(\UBMn[\row;\col]).
\end{equation*}
By taking absolute value and applying the triangle inequality,
this implies
\begin{align}
\label{eq:ubm_premaster}
|\det \UBM|
\leq \sum_{\UBk \in \{0,1\}^\ell} \sum_{(\row,\col) \in \PI_{\UBk}}
|\det \UBMd[\COMP{\row} ; \COMP{\col}] \det \UBMn[\row ; \col]|.
\end{align}

Define the weight $\UBw(z)$ by
\begin{equation}\label{eq:ub_d}
\UBw(z) \Def \prod_{i \in [\ell]\setminus \Mcol} |z-c_i|^{(1-k_i) a_i \log N}.
\end{equation}
By Lemma~\ref{lem:bidiag_det}, for any $\row,\col \subseteq [N]$ the minor $|\det \UBMd[\COMP{\row};\COMP{\col}]|$ is given by the product of its diagonal entries.  All entries of $\UBMd$ are bounded by $1.$  For those $(\row,\col) \in \PI_{\UBk},$ the diagonal entries of blocks for which $k_i = 0$ can be bounded by $|z-c_i|.$  Hence, we get the bound
\[
  |\det \UBMd[\COMP{\row};\COMP{\col}]| \leq \UBw(z) \leq 1.
\]

To complement this bound, we control the magnitude of $\det \UBMn[\row ; \col]$ over all minors.
\begin{lemma}
\label{lem:ub_noise}
For any fixed $\delta > 0,$ there is a constant $C > 0$ so that, 
with high probability,
for all $\row,\col \subseteq [N]$ with $|\row|=|\col|=k$ we have
\[
|\det \UBMn[\row ; \col]| \leq C ({k!})^{\half+\nicefrac{\delta}{\log N}}N^{(-\nu -\half)k} e^{(\log N)^C}.
\]
\end{lemma}
\begin{proof}
  We apply a union bound over all choices $(\row, \col)$ with
  $|\row| = |\col|.$  
  Let $\Omega_N$ denote the event that
  all elements of $\UBMn$ are at most $N^{-\nu-\half}\log N$ in modulus;
  because the entries are standard Gaussian, $\Pr(\Omega_N)\to 1$
  as $N\to\infty$.
  Thus, for choices of $(\row, \col)$ with $|\row | = k \leq (\log N)^3,$
  we have on $\Omega_N$ that
  \[
    |\det \UBMn[\row; \col]| \leq k! N^{(-\nu -\half) k} (\log N)^{k} \leq N^{(-\nu -\half)k} e^{ (\log N)^4}
  \]
  for all $N$ large.

  For $(\row, \col)$ with $|\row|=|\col| = k > (\log N)^3,$ we apply Lemma \ref{lem:Gdet_ub} to get that \[
    \Pr[ |\det \UBMn[\row; \col]| \geq N^{(-\nu -\half)k}(k!)^{\half + \nicefrac{\delta}{\log N}}
  ] \leq e^{-\lfloor \nicefrac{k\delta}{\log N} \rfloor^2 \log(\nicefrac{k}{2e^2})}.
  \]
  Summing over all choices of $(\row, \col),$  we have that
  \begin{align*}
    \Pr[
      \exists (\row,\col), |\row| > (\log N)^3 :
|\det \UBMn[\row; \col]| \geq N^{(-\nu - \half) |\row|}(|\row|!)^{\half + \nicefrac{\delta}{\log N}}
    ] \hspace{-3in}&\hspace{3in}\\
    &\leq \sum_{k= \lceil (\log N)^3 \rceil}^N \binom{N}{k}^2 e^{-\lfloor \nicefrac{k\delta}{\log N} \rfloor^2\log(\nicefrac{k}{2e^2})} \\
    &\leq \sum_{k= \lceil (\log N)^3 \rceil}^N N^{2k} e^{-ck^{{4}/{3}}\delta^2\log(k)} \leq e^{-O( (\log N)^4)},
  \end{align*}
  completing the proof.
\end{proof}

In light of Lemma~\ref{lem:ub_noise} and~\eqref{eq:ub_d}, it is possible to rewrite \eqref{eq:ubm_premaster} as
\begin{align}
  \nonumber
|\det \UBM|
&\leq C e^{ (\log N)^C} \sum_{\UBk \in \{0,1\}^\ell}
\UBw(z)
\sum_{(\row,\col)\in \PI_{\UBk}} ({|\row|!})^{\half+\nicefrac{\delta}{\log N}}N^{(-\nu -\half)|\row|} \\
\label{eq:ubm_master}
&\leq C e^{ (\log N)^C} \sum_{\UBk \in \{0,1\}^\ell}
\UBw(z)
\left(e^\delta N^{-\nu}\right)^{|\UBk\|_1}
\sum_{(\row,\col)\in \PI_{\UBk}}\left(e^\delta N^{-\nu}\right)^{(|\row|-\|\UBk\|_1)}.
\end{align}

Next, we apply the following estimate, whose proof we postpone to the end of
this section, to conclude the proof of Proposition~\ref{prop:ub}.
\begin{lemma}
  \label{lem:ub1}
  For any $\UBk \in \{0,1\}^\ell,$
  \[
    \sum_{(\row,\col)\in \PI_{\UBk}} \left(e^\delta N^{-\nu}\right)^{(|\row|-\|\UBk\|_1)}
    \leq e^{ 2e^{{\delta}/{2}} N^{1 - {\nu}/{2}}}.
  \]
\end{lemma}
Applying Lemma~\ref{lem:ub1} to \eqref{eq:ubm_master}, we get
\[
|\det \UBM|
\leq  C e^{ (\log N)^C}e^{ 2e^{{\delta}/{2}} N^{1 - {\nu}/{2}}}
\sum_{\UBk \in \{0,1\}^\ell}
\UBw(z)
\left(e^\delta N^{-\nu}\right)^{\|\UBk\|_1}.
\]
By our assumption that $\ell = o(N)$ we may replace the sum by a maximum:
\begin{align}
  \frac{1}{N} \log |\det \UBM| &\le o(1)
  +\max_{\UBk \in \{0,1\}^{\ell}}\frac{\log\left(\UBw(z)\left(e^\delta N^{-\nu}\right)^{\|\UBk\|_1}\right)}{N} \nonumber \\
  &= o(1) + \frac{\log N}{N} \sum_{i \in [\ell]\setminus \Mcol}\max(a_i \log|z-c_i| ,- \nu +\nicefrac{\delta}{\log N}),
  \label{eq:detU_bound}
\end{align}
where the last equality follows from \eqref{eq:ub_d}. We observe that
\[
\max(a_i\log|z-c_i| ,- \nu +\nicefrac{\delta}{\log N})=\begin{cases}
  a_i\log|z-c_i|& \text{if } g_i \le -\nicefrac{\delta}{\log N} \text{ and }\\
- \nu +\nicefrac{\delta}{\log N}& \text{if } g_i > -\nicefrac{\delta}{\log N}.
\end{cases}
\]
Writing $J\Def\{i\in [\ell]\ :\ g_i \le -\nicefrac{\delta}{\log N}\}$,
Translating~\eqref{eq:detU_bound} into:
\begin{equation}
\label{eq:detU_bound2}
  \frac{1}{N} \log |\det \UBM| \le o(1) + \frac{\log{N}}{N}\hspace{-2pt}\sum_{i\in J\setminus \Mcol}\hspace{-2pt}a_i\log|z-c_i|
  + \frac{\log{N}}{N}\hspace{-12pt}\sum_{i\in [\ell]\setminus ( J\cup\Mcol)}\hspace{-8pt}\left(- \nu +\nicefrac{\delta}{\log N}\right).
\end{equation}
Thus using~\eqref{eq:ub0} and \eqref{eq:detU_bound2},  we have shown that with probability going to $1$,
for $\delta<\nu$, we have
\begin{align}
  \nonumber
  \frac{1}{N}\log |\det \model - zI| &= \frac{\log |\det \UBM|}{N}+\frac{\log{N}}{N}\sum_{i \in \Mcol}{a_i}
  \log|z-c_i| \nonumber\\
  &= o(1) + \frac{\log{N}}{N}\hspace{-2pt}\sum_{i\in J}\hspace{-2pt}a_i\log|z-c_i|
  + \frac{\log{N}}{N}\hspace{-2pt}\sum_{i\in [\ell] \setminus J}\hspace{-4pt}\left(-
  \nu +\nicefrac{\delta}{\log N}\right) \nonumber\\
  \label{eq:ub2}
  &\leq
  o(1) + \frac{\log{N}}{N}\hspace{-2pt}\sum_{i\in J}\hspace{-2pt}a_i\log|z-c_i|
  + \frac{\log{N}}{N}\hspace{-2pt}\sum_{i\in [\ell] \setminus J}\hspace{-4pt}-
  \nu.
\end{align}
Where the second equality uses the fact that $g_i<-\nu<-\delta$ for $i\in \Mcol$, and the inequality
uses the assumption that $\ell=o(N)$.

Rewriting \eqref{eq:target}, we have the following bound on $\LPL(z):$
  \[
    \LPL(z) = \frac{\log N}{N}\sum_{i : g_i < 0} a_i \log |z-c_i| - \frac{\log N}{N}\sum_{i : g_i \geq 0}\nu.
  \]
  The bound given in \eqref{eq:ub2} differs only in that some terms for which $-\nicefrac{\delta}{\log N} < g_i \leq 0$ have been moved from the second sum to the first.
  Thus \eqref{eq:ub2} can be rewritten as
  \[
    \LP(z) \leq \LPL(z)  + o(1) + \frac{\log N}{N}\sum_{\substack{i\in [\ell]\setminus J\\
    g_i\le0 }} -g_i.
  \]
  As $[\ell]=o(N)$ and $g_i>-\nicefrac{\delta}{\log N}$ for all $i\in[\ell]\setminus J$, we get
  \[
    \LP(z) \leq \LPL(z)  + o(1),
  \]
 as required.  \qed

  \begin{proof}[Proof of Lemma~\ref{lem:ub1}]
    Let $\UBm = (m_i)_{i \in [\ell]}$ denote an element of the set
    \[
      \UBconf \Def [a_1\log N] \times [a_2\log N] \times \cdots \times [a_\ell \log N],
    \]
    and let $\NS \subset \PI_{\UBk}$ denote the collection of pairs $(\row, \col)$ so that $m_i=|\row_i|=|\col_i|.$  Note that this forces $m_i = 0$ for all those $i \in [\ell]$ so that $k_i = 0.$

    The cardinality of $\NS$ is given by
    \[
      |\NS|  = \prod_{i : k_i = 1} \binom{a_i \log N}{m_i}^2.
    \]
        We may then use this to obtain the bound,
    \begin{align}
      \nonumber
      \hspace{18 pt}\sum_{(\row,\col)\in \PI_{\UBk}} \left(e^\delta N^{-\nu}\right)^{(|I|-\|\UBk\|_1)}
      &=
      \sum_{j = \|\UBk\|_1}^{N}
      \sum_{\substack{ \UBm \in \UBconf \\ \|\UBm\|_1 = j }}
      |\NS|\left(e^\delta N^{-\nu}\right)^{(j-\|\UBk\|_1)} \\
      \nonumber
      &=
      \sum_{j = \|\UBk\|_1}^{N}
      \sum_{\substack{ \UBm \in \UBconf \\ \|\UBm\|_1 = j }}
      \prod_{i : k_i = 1}
      \binom{a_i \log N}{m_i}^2
      \left(e^\delta N^{-\nu}\right)^{(m_i - 1)} \\
      \nonumber
      &=
      \prod_{i : k_i = 1}
      \biggl[
	\sum_{m_i = 1}^{a_i\log N}
	\binom{a_i \log N}{m_i}^2
	\left(e^\delta N^{-\nu}\right)^{(m_i - 1)}
      \biggr]\\
      \nonumber
      &\leq
      \prod_{i : k_i = 1}
      \biggl[
	\sum_{m_i = 0}^{a_i\log N}
	\binom{a_i \log N}{m_i}^2
	\left(e^\delta N^{-\nu}\right)^{m_i}
      \biggr].\\
    \intertext{As $\sum_{m_i = 0}^{a_i\log N}
	\binom{a_i \log N}{m_i}^2
	\left(e^\delta N^{-\nu}\right)^{m_i}>1$ for all $i\in [\ell]$ we can complete the product to get:}
      \hspace{20pt}\sum_{(\row,\col)\in \PI_{\UBk}} \left(e^\delta N^{-\nu}\right)^{(|I|-\|\UBk\|_1)}
      &\leq
      \prod_{i = 1}^\ell
      \biggl[
	\sum_{m_i = 0}^{a_i\log N}
	\binom{a_i \log N}{m_i}^2
	\left(e^\delta N^{-\nu}\right)^{m_i}
      \biggr].  \label{eq:ub3} \raisetag{-25pt}
    \end{align}
    For any $q > 0$ and any $t \in \N,$ define the polynomial
    \[
      P_t(q) = \sum_{m=0}^t \binom{t}{m}^2 q^m.
    \]

    Thus, in terms of~\eqref{eq:ub3}, we have
    \begin{equation}
      \label{eq:ub4}
      \sum_{(I,J)\in \PI_{\UBk}} \left(e^\delta N^{-\nu}\right)^{(|I|-\|\UBk\|_1)}
      \leq \prod_{i=1}^\ell
      P_{a_i\log N}\left(e^\delta N^{-\nu}\right).
    \end{equation}

    We now bound $P_t(q),$ by comparing with Taylor series:
    \begin{align*}
      P_t(q)
      &=
      \sum_{m=0}^t \binom{t}{m}^2 q^m
      \leq
      \sum_{m=0}^t \frac{t^{2m}}{(m!)^2}q^m
      \leq
      \sum_{m=0}^t \frac{t^{2m}}{(2m)!}q^m \binom{2m}{m} \\
      &\leq
      \sum_{m=0}^\infty \frac{t^{2m}}{(2m)!}q^m 2^{2m}
      \leq
      \sum_{m=0}^\infty \frac{(2t\sqrt{q})^m}{m!} = e^{2t\sqrt{q}}.
    \end{align*}

    Combining this bound with~\eqref{eq:ub4}, we get that
    \[
      \sum_{(I,J)\in \PI_{\UBk}} \left(e^\delta N^{-\nu}\right)^{(|I|-\|\UBk\|_1)}
      \leq \exp\left(2 \left(e^\delta N^{-\nu}\right)^{\half} \sum_{i=1}^\ell a_i \log N\right).
    \]
    As $\sum_{i=1}^\ell a_i \log N = N,$ the proof is complete.

  \end{proof}
\section{Lower bound}
\label{sec:lb}

This section is dedicated to the proof of the following proposition.
\begin{proposition}
  \label{prop:lb}
  For $z \in \GOODz, \ell=o(N)$ we have that for all $\delta > 0,$
  \[
    \lim_{N\to\infty}\Pr[
      \LP(z) \leq \LPL(z) - \delta
    ] = 0.
  \]
\end{proposition}

Given $\delta>0$ small enough, we seek to estimate the absolute value of the determinant of $\model-zI$ from below. As 
reordering 
rows and columns
does not change the absolute value of the determinant, we first apply such operations to bring the matrix to a convenient form.

Recall that
\(
g_i = (-\nu) - a_i \log|z-c_i|.
\)
Order the blocks such that all the blocks which satisfy both
\begin{enumerate}
  \item $g_i \geq -\delta\left(a_i+\frac{N}{\ell\log(N)}\right)$
  \item $|z-c_i| \leq 1$
\end{enumerate}
appear first, and let $Q$ be the index of the last $i$ for which this is so. For each $i\in [Q]$, in increasing order, move the first column of the $i$-th block to the the far right, and its last row to the far bottom.  Then, reverse the order of the indices of each block, which can also be achieved by conjugating by the appropriate permutation matrix.  See Figure~\ref{fig:monalisa} for
an illustration of this procedure.

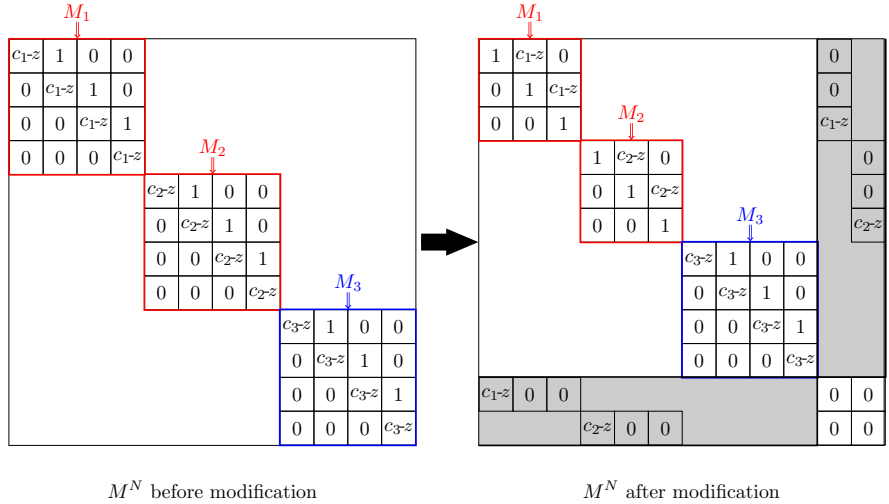
\begin{figure}[b]
\adjustbox{scale=0.75}{
\begin{tikzpicture}
\tikzset{square matrix/.style={
    matrix of nodes,
    column sep=-\pgflinewidth, row sep=-\pgflinewidth,
    nodes={draw,
      minimum height=#1,
      anchor=center,
      text width=#1,
      align=center,
      inner sep=0pt
    },
  },
  square matrix/.default=0.6cm
}

\tikzstyle{myarrows}=[line width=1mm,draw=black,-triangle 45,postaction={draw, line width=3mm, shorten >=4.5mm, -}]

\matrix(A)[draw,square matrix,inner sep=0mm]
{
$c_1$\hspace{-1pt}\small{-}\hspace{-1pt}$z$ &1  & 0 & 0 & & & & & & & &\\
0 & $c_1$\hspace{-1pt}\small{-}\hspace{-1pt}$z$ & 1 & 0 & & & & & & & &\\
0 & 0 & $c_1$\hspace{-1pt}\small{-}\hspace{-1pt}$z$ & 1 & & & & & & & &\\
0 & 0 & 0 & $c_1$\hspace{-1pt}\small{-}\hspace{-1pt}$z$ & & & & & & & &\\
& & & &$c_2$\hspace{-1pt}\small{-}\hspace{-1pt}$z$ &1   & 0 & 0 & & & &\\
& & & &0 & $c_2$\hspace{-1pt}\small{-}\hspace{-1pt}$z$ & 1 & 0 & & & &\\
& & & &0 & 0 & $c_2$\hspace{-1pt}\small{-}\hspace{-1pt}$z$ & 1 & & & &\\
& & & &0 & 0 & 0 & $c_2$\hspace{-1pt}\small{-}\hspace{-1pt}$z$ & & & &\\
& & & & & & & &$c_3$\hspace{-1pt}\small{-}\hspace{-1pt}$z$ & 1 & 0 & 0\\
& & & & & & & &0 & $c_3$\hspace{-1pt}\small{-}\hspace{-1pt}$z$ & 1 & 0\\
& & & & & & & &0 & 0 & $c_3$\hspace{-1pt}\small{-}\hspace{-1pt}$z$ & 1\\
& & & & & & & &0 & 0 & 0& $c_3$\hspace{-1pt}\small{-}\hspace{-1pt}$z$ \\
} node [below =0.5cm of A] {$M^N$ before modification};
\draw[thick,color=red] (A-1-1.north west) -- (A-1-4.north east) -- (A-4-4.south east) -- (A-4-1.south west) -- (A-1-1.north west);
\draw[thick,color=red] (A-5-5.north west) -- (A-5-8.north east) -- (A-8-8.south east) -- (A-8-5.south west) -- (A-5-5.north west);
\draw[thick,color=blue] (A-9-9.north west) -- (A-9-12.north east) -- (A-12-12.south east) -- (A-12-9.south west) -- (A-9-9.north west);
\draw[color=red,double,implies-](A-1-2.north east) -- +(0,0.3) node [pos=0.66,above] {$M_1$};
\draw[color=red,double,implies-](A-5-6.north east) -- +(0,0.3) node [pos=0.66,above] {$M_2$};
\draw[color=blue,double,implies-](A-9-10.north east) -- +(0,0.3) node [pos=0.66,above] {$M_3$};

\matrix(B)[right = 1.1cm of A,draw,square matrix,inner sep=0mm]
{
1 & $c_1$\hspace{-1pt}\small{-}\hspace{-1pt}$z$ &0  & &  &  & & & & &0\\
0 & 1 &$c_1$\hspace{-1pt}\small{-}\hspace{-1pt}$z$ &  &  &  & & & & &0\\
0 & 0 & 1 & &  & & & & & & $c_1$\hspace{-1pt}\small{-}\hspace{-1pt}$z$\\
& & &1&$c_2$\hspace{-1pt}\small{-}\hspace{-1pt}$z$   & 0 &  & & & & & 0 \\
& & &0 & 1&$c_2$\hspace{-1pt}\small{-}\hspace{-1pt}$z$  &   & & & & & 0\\
& & &0 & 0 & 1 & & & & & & $c_2$\hspace{-1pt}\small{-}\hspace{-1pt}$z$\\
& & & & &  &$c_3$\hspace{-1pt}\small{-}\hspace{-1pt}$z$ & 1 & 0 & 0& & & &\\
& & & & &  &0 & $c_3$\hspace{-1pt}\small{-}\hspace{-1pt}$z$ & 1 & 0& & & & \\
& & & & &  &0 & 0 & $c_3$\hspace{-1pt}\small{-}\hspace{-1pt}$z$ & 1& & & & \\
& & & & &  &0 & 0 & 0& $c_3$\hspace{-1pt}\small{-}\hspace{-1pt}$z$ & & & & \\
$c_1$\hspace{-1pt}\small{-}\hspace{-1pt}$z$ &0&0 & & &  & &  & &  & 0 & 0 \\
& & & $c_2$\hspace{-1pt}\small{-}\hspace{-1pt}$z$ &0&0&  & & &   & 0 &0 \\
} node [below= 0.5cm of B, align=flush center] {$M^N$ after modification};

\draw[thick,color=red] (B-1-1.north west) -- (B-1-3.north east) -- (B-3-3.south east) -- (B-3-1.south west) -- (B-1-1.north west);
\draw[thick,color=red] (B-4-4.north west) -- (B-4-6.north east) -- (B-6-6.south east) -- (B-6-4.south west) -- (B-4-4.north west);
\draw[thick,color=blue] (B-7-7.north west) -- (B-7-10.north east) -- (B-10-10.south east) -- (B-10-7.south west) -- (B-7-7.north west);
\draw[color=red,double,implies-](B-1-2.north) -- +(0,0.3) node [pos=0.66,above] {$M_1$};
\draw[color=red,double,implies-](B-4-5.north) -- +(0,0.3) node [pos=0.66,above] {$M_2$};
\draw[color=blue,double,implies-](B-7-8.north east) -- +(0,0.3) node [pos=0.66,above] {$M_3$};
\draw[thick,color=black,fill=black, opacity=0.2] (B-1-11.north west) -- (B-11-11.north west) -- (B-11-12.north east) -- +(0,6);
\draw[thick,color=black,fill=black, opacity=0.2] (B-11-1.north west) -- (B-11-11.north west) -- (B-12-11.south west) -- +(-6,0);
\draw[thick,color=black] (B-1-11.north west) -- (B-11-11.north west) -- (B-11-12.north east) -- +(0,6);
\draw[thick,color=black] (B-11-1.north west) -- (B-11-11.north west) -- (B-12-11.south west) -- +(-6,0);
\draw [myarrows](A){}+(3.7,0)--(B){};

\end{tikzpicture}}
\caption{An illustration
	of how $M^N-zI$ is modified, with $Q=2.$ The submatrix containing $M_1,M_2,M_3$ is the matrix $A$ in \eqref{eq:LBblock}.}
  \label{fig:monalisa}
\end{figure}

With a slight abuse of notation, throughout this section we keep calling the remaining matrix $\model.$  We label this matrix
\begin{equation}
  \label{eq:LBblock}
\model =
\begin{bmatrix}
A+N^{-\half-\nu}G_{N-Q} & B \\
C & N^{-\half-\nu}G_{Q} \\
\end{bmatrix},
\end{equation}
where $A$ is deterministic, $G_{N-Q}$ and $G_{Q}$ are independent matrices of
noise of size $N-Q\times N-Q$ and $Q\times Q$ respectively. To compute the
determinant of $\model -zI,$ we begin by computing the determinant of
$A+N^{-\half-\nu}G_{N-Q}-zI$. The Schur complement
formula is then used to calculate the whole determinant.

Write $\LBMd$, and $\LBMn$ for the matrices $A-zI$ and
$N^{-\half-\nu}G_{N-Q}$, denoting the $i$-th block of
$\LBMd$  by $\LBMd_i$. As the determinant of $\LBM= \LBMd + \LBMn$ will be bounded using the
second moment method, we begin by calculating $\Exp\det(\LBM)$. Applying \eqref{eq:det_decomposition} to $\LBMd$ and $\LBMn$  implies:
\begin{equation}\label{eq:LBdet_decomposition}
\det(\LBM) = \sum_{\substack{\row,\col \subset [N-Q] \\ |\row|=|\col|}} (-1)^{s(\row,\col)} \det(\LBMd[\COMP{\row};\COMP{\col}]) \det(\LBMn[\row;\col]).
\end{equation}

For any $\row$ or $\col \neq \emptyset$, we have that $\Exp \det(\LBMn[\row;\col])=0.$  Thus, by the linearity of expectation, may deduce that
\begin{align*}
\Exp\det(\LBM)
&= \sum_{\substack{\row,\col \subset [N-Q] \\ |\row|=|\col|}} (-1)^{s(\row,\col)} \det(\LBMd[\COMP{\row};\COMP{\col}]) \Exp\det(\LBMn[\row;\col])  \\
&= \Exp \det(\LBMd) \\
&= \prod_{i\notin [Q]} (z-c_i)^{a_i\log N}.
\end{align*}

Given that $\Exp\det(\LBM)=\Exp \det(\LBMd)$, the following claim would allow us to apply Chebyshev's inequality and prove Proposition~\ref{prop:lb}.
\begin{lemma}
  \label{lem:LB_var}
  Given $\ell=o(N)$, for $z \in \GOODz$, we have that there is a $C=C(\nu) > 0$ so that for all $\delta > 0$ there is an $N_0 = N_0(\nu,\delta)$ so that
  \[
    \Var \det(\LBM)/|\det(\LBMd)|^2 < C \cdot \delta
  \]
  for $N \geq N_0.$
\end{lemma}

In the proof of Lemma~\ref{lem:LB_var} we shall make use of the following bound whose proof we give after the proof of Lemma~\ref{lem:LB_var}.
\begin{lemma}
\label{lem:minor_sum}
For $z \in \GOODz$ and $1 > \delta > 0,$
there is an $N_0$
depending on
$\nu$, $z$ and $\delta$, so that for any $i \in [\ell]$ and any
$\row_i \subseteq [\dim \LBMd_i],$
\[
  \sum_{|\col_i|=|\row_i|} |\det(\LBMd_i[\COMP{\row_i};\COMP{\col_i}])|^2N^{-2\nu|\row_i|}<|\det(\LBMd_i)|^2\delta^{|\row_i|}
\]
for all $N \geq N_0.$
\end{lemma}

\begin{proof}[Proof of Lemma~\ref{lem:LB_var}]
For any $X_i,Y_i \subseteq [N-Q], i=1,2$ define
\begin{multline*}
  K(X_0,Y_0,X_1,Y_1)  \\
=\Cov(
\det(\LBMd[\COMP{\row_0};\COMP{\col_0}]) \det(\LBMn[\row_0;\col_0])
,\det(\LBMd[\COMP{\row_1};\COMP{\col_1}]) \det(\LBMn[\row_1;\col_1])).
\end{multline*}
Using \eqref{eq:LBdet_decomposition}, we reformulate the above as
\begin{align}\label{eq:cov_decomposition}
\Var \det(\LBM) =
\sum_{\substack{\row_i,\col_i \subset [N-Q] \\ |\row_i|=|\col_i| \\ i=1,2}}
(-1)^{s(\row_0,\col_0)+s(\row_1,\col_1)}
K(X_0,Y_0,X_1,Y_1).
\end{align}
Observe that if either $\row_0 \neq \row_1$ or $\col_0 \neq \col_1,$ then
\(
  K(X_0,Y_0,X_1,Y_1) =0.
\)
Using this \eqref{eq:cov_decomposition} reduces to
\begin{align*}
\Var \det(\LBM)
&=\sum_{\substack{\row,\col \subset [N-Q] \\ |\row|=|\col|}}
|\det(\LBMd[\COMP{\row};\COMP{\col}])|^2\Var(\det(\LBMn[\row;\col])) \\
&=\sum_{\substack{\row,\col \subset [N-Q] \\ |\row|=|\col|>0}}
|\det(\LBMd[\COMP{\row};\COMP{\col}])|^2N^{-2(\nu+\half)|\row|}(|\row|!)
\end{align*}

Recall that $\row_i$ are the rows of $\row$ which intersect the block $\LBMd_i$ and $\col_i$ are the columns of $\row$ which intersect the block $\LBMd_i.$ This allows further development of \eqref{eq:cov_decomposition} into

\begin{align}
\nonumber
\Var \det(\LBM)
&=\sum_{\substack{\row \subset [N-Q] \\ |\row|>0}}N^{-|\row|}(|\row|!)
\prod_{i=1}^\ell \sum_{|\col_i|=|\row_i|} |\det(\LBMd_i[\COMP{\row_i};\COMP{\col_i}])|^2N^{-2\nu|\row_i|}\\
\nonumber
&=\sum_{k=1}^{N-Q}N^{-k}k!\sum_{\substack{\row \subset [N-Q] \\ |\row|=k}}
\prod_{i=1}^\ell \sum_{|\col_i|=|\row_i|} |\det(\LBMd_i[\COMP{\row_i};\COMP{\col_i}])|^2N^{-2\nu|\row_i|}\\
\label{eq:var1}
&\le\sum_{k=1}^{N-Q}\max_{\substack{\row \subset [N-Q] \\ |\row|=k}}
\prod_{i=1}^\ell \sum_{|\col_i|=|\row_i|} |\det(\LBMd_i[\COMP{\row_i};\COMP{\col_i}])|^2N^{-2\nu|\row_i|}.
\end{align}

Combining \eqref{eq:var1} with Lemma \ref{lem:minor_sum}, we have
\begin{align}
\nonumber
\Var \det(\LBM)
&\le\sum_{k=1}^{N-Q}\max_{\substack{\row \subset [N-Q] \\ |\row|=k}}
\prod_{i=1}^\ell \sum_{|\col_i|=|\row_i|} |\det(\LBMd_i[\COMP{\row_i};\COMP{\col_i}])|^2N^{-2\nu|\row_i|}\\
\nonumber
&\le\sum_{k=1}^{N-Q}\max_{\substack{\row \subset [N-Q] \\ |\row|=k}}
\prod_{i=1}^\ell |\det(\LBMd_i)|^2{\delta}^{|\row_i|} \\
\nonumber
&\le
|\det(\LBMd)|^2
\sum_{k=1}^{N-Q} \delta^{k}\le |\det(\LBMd)|^2 \frac{\delta}{1-\delta}.
\end{align}
\end{proof}

We now turn to proving Lemma~\ref{lem:minor_sum}.
\begin{proof}[Proof of Lemma~\ref{lem:minor_sum}]
  We divide the proof into three cases, according to the type of block
  $\LBMd_i.$  Define a partition of $[\ell] = S_1 \cup S_2 \cup S_3$ by
\begin{align*}
  S_1 &\Def \left\{ i :  g_i \geq -\delta\left(a_i+\frac{N}{\ell\log N}\right), |z-c_i| \leq 1 \right\}, \\
  S_2 &\Def \left\{ i :  g_i < -\delta\left(a_i+\frac{N}{\ell\log N}\right), |z-c_i| \leq 1 \right\}, \\
  S_3 &\Def \left\{ i :  |z-c_i| > 1 \right\}.
\end{align*}
Note that $S_1 = [Q].$

Set $k=|\row_i|.$  Let $\row_i = \left\{ x_1, x_2, \ldots, x_k \right\},$ with $x_1 < x_2 < \cdots < x_k.$
By virtue of Lemma~\ref{lem:block_calc}, the only choices of $\col_i = \left\{ y_1 , y_2, \ldots y_k \right\}$ with $y_1 < y_2 < \cdots < y_k$ which we need to consider are those which satisfy
\[
  x_1 \leq y_1 < x_2 \leq y_2 < \cdots < x_{k} \leq y_k.
\]

\noindent \emph{ The case $i\in S_1$:}
For $i \in S_1,$ we have that $\LBMd_i$ has $1$ on the diagonal and $c_i$ on the superdiagonal.  Thus, we may write
\[
T_i \Def \sum_{|\col_i|=|\row_i|} |\det(\LBMd_i[\COMP{\row_i};\COMP{\col_i}])|^2
=
\prod_{j=1}^k \sum_{r=0}^{x_{j+1} - x_j-1} |z-c_i|^{2r},
\]
where we take $x_{k+1} = \dim \LBMd_i \leq a_i \log N.$
%
%
This we may control either by bounding each element by $1$ or by bounding the truncated geometric series by the entire geometric series, which implies
\[
  T_i^{1/k} \leq \min \left\{ a_i \log N, \frac{1}{1-|z-c_i|^2} \right\} \leq N^{2\nu - \BADeps},
\]
where the rightmost inequality follows from the fact that $z \in \GOODz.$

\noindent \emph{ The case $i\in S_3$:}
This is nearly identical to the case $i \in S_1,$ and so we show it first.  For $i \in S_3,$ it is now the case that $\LBMd_i$ has $c_i-z$ on the diagonal and $1$ on the superdiagonal.  Pulling out a factor of $c_i-z,$ we essentially reduce the determinant to the previous case, i.e.
\begin{align}
  \nonumber
  T_i &= |z-c_i|^{2|\COMP{\row_i}|} \sum_{|\col_i|=|\row_i|} |\det((z-c_i)^{-1}\LBMd_i[\COMP{\row_i};\COMP{\col_i}])|^2 \\
  \nonumber
  &= |z-c_i|^{2|\COMP{\row_i}|}\prod_{j=1}^k \sum_{r=0}^{x_{j+1} - x_j-1} |z-c_i|^{-2r} \\
  \nonumber
  &= |z-c_i|^{2(a_i\log N - k)}\prod_{j=1}^k \sum_{r=0}^{x_{j+1} - x_j-1} |z-c_i|^{-2r} \\
  \label{eq:var2}
  &\leq |z-c_i|^{2(a_i\log N - k)}\biggl[\sum_{r=0}^{a_i \log N-1} |z-c_i|^{-2r}\biggr]^k.
\end{align}
Thus, on the one hand, we may bound the sum by bounding each term by $1,$ or we may bound by a geometric series.  In the first case we get
\[
  T_i\hspace{-2pt}
  \leq \hspace{-2pt}|z-c_i|^{2(a_i\log N - k)} |a_i \log N|^{k}\hspace{-2pt}
  \leq  \hspace{-2pt}|z-c_i|^{2a_i\log N} |a_i \log N|^{k}\hspace{-2pt}
  = |\det(\LBMd_i)|^2 |a_i \log N|^{k}.
\]
In the second case we get
\[
  T_i \leq |z-c_i|^{2(a_i\log N - k)}\biggl[\frac{1}{1-|z-c_i|^{-2}}\biggr]^k = |\det(\LBMd_i)|^2\biggl[\frac{1}{|z-c_i|^2-1}\biggr]^k.
\]
Thus as $z \in \GOODz$, it follows that
\[
  T_i \leq|\det(\LBMd_i)|^2 N^{(2\nu - \BADeps)k}.
\]

\noindent \emph{ The case $i\in S_2$:}
As in the previous case, for $i \in S_2$, the blocks $\LBMd_i$ have $c_i-z$ on the diagonal and $1$ on the superdiagonal. From Lemma~\ref{lem:bidiag_det}, we have that
\(
|\det(\LBMd_i[\COMP{\row_i};\COMP{\col_i}])|^2 \leq 1,
\)
and hence
\begin{equation}\label{eq:s2_0}
  T_i \leq (a_i \log N)^k.
\end{equation}

Note that the condition that $g_i < -\delta a_i$ and $|z-c_i| \leq 1$ imposes a lower bound on the size of $a_i,$ namely
\[
-\delta a_i > g_i = -\nu - a_i \log|z-c_i| > - \nu
\]
and hence $a_i < \nicefrac{\nu}{\delta}.$

Let $\xi_\delta < \infty $ be the solution of $\delta \xi_\delta = \sup_{x \geq 0} xe^{-2\delta x}.$  We get that
\[
  a_i \log N \leq \delta \xi_\delta \cdot e^{2\delta a_i \log N}
\]
As $\ell = o(N),$ there is an $N_1 = N_1(\delta)$ sufficiently large so that $\xi_\delta \leq e^{2\delta \nicefrac{N}{\ell}}$ for all $N \geq N_1.$ Hence, for all $N \geq \max(N_0,N_1),$
\begin{equation}\label{eq:s2_2}
  a_i \log N
  \leq \delta e^{2\delta a_i \log N + 2\delta \nicefrac{N}{\ell}}
  \leq \delta e^{-2g_i \log N} = \delta N^{2\nu} |z-c_i|^{2 a_i \log N}= \delta N^{2\nu} | \det \LBMd_i |^2.
\end{equation}
Combining \eqref{eq:s2_0} and \eqref{eq:s2_2}, it follows that
\[
  T_i \leq (a_i \log N)^k \leq \delta^k N^{2\nu k} | \det \LBMd_i |^{2k}\leq \delta^k N^{2\nu k} | \det \LBMd_i |^{2}
\]
for all $N$ sufficiently large.

\end{proof}

Having controlled the determinant of $A+N^{-\half-\nu}G_{N-Q},$ it remains
to show that the determinant of $\model-zI$ is not too small.  Our proof rests
on the following stochastic domination lemma.
\begin{lemma}
  \label{lem:sd}
  Suppose that $E$ is a $Q \times Q$ standard Gaussian matrix.  Then for any $Q \times Q$ matrix $M$ independent of $E$ and all $t \geq 0,$
  \[
    \Pr[ |\det(E + M)| \leq t] \leq \Pr[ |\det E| \leq t].
  \]
\end{lemma}
\begin{proof}
  The proof rests on the parallelpiped formula for the modulus of a determinant.  For any square matrix $M',$ let $S_j(M')$ be the span of the first $j$ columns.  Then we have the following identity
  \[
    |\det(E+M)|^2 = \prod_{j=1}^Q \left|\proj_{(S_{j-1}(E+M))^{\perp}}\left(  (E+M)_j \right)\right|^2,
  \]
  where $(E+M)_j$ denotes the $j$-th column of $E+M.$
  By virtue of absolute continuity with lebesuge measure, we have that each $S_j(E+M)$ is almost surely $j$-dimensional.
  Conditioned on columns $1,2,\ldots,j$ and choosing an orthonormal basis for $S_{j-1}(E+M)^{\perp},$ the projection
  $\proj_{(S_{j-1}(E+M))^{\perp}}\left(  (E+M)_j \right)$ has the law of an uncentered $(Q-j+1)$-dimensional standard Gaussian vector.  The norm of an uncentered standard Gaussian vector stochastically dominates the norm of a centered standard Gaussian vector, and hence we have the relation
  \begin{multline*}
    \Pr[
      \left|\proj_{(S_{j-1}(E+M))^{\perp}}\left(  (E+M)_j \right)\right|^2 \leq t ~\bigl\vert~ \sigma(S_{j-1}(E+M))
    ] \\
    \leq
    \Pr[
      \left|\proj_{(S_{j-1}(E))^{\perp}}\left(  (E)_j \right)\right|^2 \leq t ~\bigl\vert~ \sigma(S_{j-1}(E))
    ].
  \end{multline*}
  Hence, applying this relation iteratively, we get the desired result that
  \[
    \Pr[ |\det(E + M)|^2 \leq t] \leq \Pr[ |\det E|^2 \leq t].
  \]

\end{proof}

With this result in hand, we finally prove to lower bound proposition using the Schur complement formula and Chebyshev's inequality.
\begin{proof}[Proof of Proposition~\ref{prop:lb}]
  Recall that
  \[
    \LP(z) = \frac{1}{N} \log | \det( \model -zI) |.
  \]
  We will apply the Schur complement formula to compute this determinant, using the block structure from \eqref{eq:LBblock}.  For $\delta>0$ sufficiently small, Lemma~\ref{lem:LB_var} and Chebyshev's inequality imply that
  \[
    |\det(A+N^{-\nu-\half}G_{N-Q} - zI)| \geq \half |\det(\LBMd)|
  \]
  with probability going to $1.$  As a corollary, the corner submatrix $A+N^{-\nu-\half}G_{N-Q} - zI$ is invertible with high probability.
  Hence, we may apply the Schur complement formula to write
  \[
    \det( \model -zI) = \det(A+N^{-\nu-\half}G_{N-Q} - zI) \det(N^{-\nu-\half}G_{Q} - Z),
  \]
  with $Z$ some $Q \times Q$ matrix independent of $G_{Q}.$  Applying Lemma~\ref{lem:sd} and Lemma~\ref{lem:Gdet_lb}, we have that
  \[
    |\det(N^{-\nu-\half}G_{Q} - Z)| \geq N^{(-\nu-\half)Q}(Q!)^{1/2}e^{-cQ}
  \]
  with high probabilty.

By our assumption that $\ell=o(N)$, we get
  \begin{align*}
    \LP(z) &\geq \frac{1}{N} \log |\det(\LBMd)| -(\nu+\half) Q\frac{\log N}{N} + \frac{1}{2}\frac{Q(\log Q)}{N} - o(1)\nonumber\\
           &=    \frac{1}{N} \log |\det(\LBMd)| -\nu Q\frac{\log N}{N} + \frac{1}{2} \frac{Q(\log Q-\log N)}{N} - o(1)\\
           &=    \frac{1}{N} \log |\det(\LBMd)| -\nu Q\frac{\log N}{N} - o(1)\\
  \end{align*}

  As $\det(\LBMd)$ is given by the product of its diagonal, we get
  \begin{align}
    \label{eq:var3}
    \LP(z)
    &\geq \frac{\log N}{N}\sum_{i > Q} a_i \log |z-c_i| - \frac{\log N}{N}\sum_{i \in [Q]}\nu  - o(1).
  \end{align}

  Recall that we wish to bound $\LP(z)$ below by $\LPL(z),$ which is given by \eqref{eq:target}:
  \[
    \LPL(z) = \frac{\log N}{N}\sum_{i : g_i < 0} a_i \log |z-c_i| - \frac{\log N}{N}\sum_{i : g_i \geq 0} \nu.
  \]

  Observing that $i\in[Q]$ implies $g_i < 0$ we get

  \begin{align*}
    \LP(z) - \LPL(z) &= \frac{\log N}{N}\sum_{i \notin [Q], g_i<0} a_i \log |z-c_i| - \frac{\log N}{N}\sum_{i \notin [Q], g_i<0} \nu -  o(1)\nonumber \\
                     &= \frac{\log N}{N}\sum_{i \notin [Q], g_i<0} g_i+o(1)\nonumber \\
                     &< \frac{\log N}{N}\sum_{i \notin [Q], g_i<0} \delta \frac{N}{\ell\log N}+o(1)\nonumber \\
                     &< \delta + o(1), \nonumber
  \end{align*}
  with the $o(1)$ errors going to $0$ at some absolute rate.
  Thus, we have that for all fixed $\delta,$
  \(
    \LP(z) \geq \LPL(z) - \delta
  \)
  with probability going to $1$ as $N \to \infty.$

\end{proof}

\section{Proof of Theorem \ref{thm:main1}}
\label{sec:final}

From Propositions~\ref{prop:ub} and \ref{prop:lb}, and the assumption that $\mu_N \weakto \mu,$ we have that for almost every $z\in \C,$ $\LP(z) \Pto \LP[\mu](z).$  It follows from Theorem 2.8.3 of \cite{TaoBook} that $\ESD$ converges to $\mu$ vaguely in probability, i.e. for all compactly supported continuous functions $\phi : \C \to \R,$
\[
	\int_\C \phi(x) d\ESD(x) \Pto \int_\C \phi(x) d\mu(x).
\]

To additionally conclude 
that 
the convergence holds in the weak topology, it suffices to show that 
there is a compact 
$K_1 \subseteq \C$ so that $\ESD(K_1^c) \Pto 0.$  Note that as we assume there is a compact $K \subseteq \C$ so that $\mu_N$ is supported on $K,$ it follows there is an $C$ sufficiently large so that all $c_i=c_i(N)$ where $1\leq i \leq \ell=\ell(N)$ and $N$ runs over $\N$ are bounded in modulus by 
 $C.$  Hence,
\begin{align*}
	\|\model\|_{\operatorname{op}}
	&\leq \|M^N + N^{-\nu-\half}G^N\|_{\operatorname{op}} \\
	&\leq \|M^N\|_{\operatorname{op}}+\|N^{-\nu-\half}G^N\|_{\operatorname{op}}
	\leq 1 + C +\|N^{-\nu-\half}G^N\|_{\operatorname{op}}.
      \end{align*}
It is easily checked that $\|G^N\|_{\operatorname{op}} \leq C' \sqrt{N}$
with high probability, for some sufficiently large $C'$.
(In fact, taking $C=1+\delta$ with arbitrary $\delta>0$ suffices.)
Since $\nu>0$,  the existence of the claimed $K_1$ follows and completes the
proof of the theorem.
\section*{Acknowledgements.}
We thank Naomi Feldheim, Mark Rudelson and Misha Sodin for useful suggestions.
\bibliographystyle{hep}

\end{document}